\documentclass[11pt]{article}
\usepackage{amsfonts}
\usepackage{amsmath}
\usepackage{amssymb}
\usepackage{eucal}
\newtheorem{theorem}{Theorem}[section]

\newtheorem{corollary}{Corollary}[theorem]

\newtheorem{definition}[theorem]{Definition}
\newtheorem{example}[theorem]{Example}

\newtheorem{lemma}[theorem]{Lemma}

\newtheorem{proposition}[theorem]{Proposition}
\newtheorem{remark}[theorem]{Remark}

\newenvironment{proof}[1][Proof]{\textbf{#1.} }

\begin{document}

\baselineskip=16pt
\title{A non commutative generalization BL-rings}

\author{{Surdive Atamewoue Tsafack$^{1,2}$, \ \ Arnaud Fobasso Tchinda$^{2}$,\ \ Yuming Feng$^{1,3,}$\thanks{Corresponding author, Email: yumingfeng25928@163.com, ymfeng@sanxiau.edu.cn},\ \ Selestin Ndjeya$^{4}$} \\
{\small $1.$   Key Laboratory of Intelligent Information Processing and Control, }\\
{\small Chongqing Three Gorges University, Wanzhou, Chongqing, 404100, China}\\
{\small $2.$ Department of Mathematics, University of Yaounde 1, Cameroon}\\
{\small $3.$ Chongqing Engineering Research Center of Internet of Things and Intelligent Control Technology ,}\\
{\small  Chongqing Three Gorges University, Wanzhou, Chongqing, 404100, China}\\
{\small $4.$ Department of Mathematics, Higher Teacher Training College, University of Yaounde 1}\\
{\small Cameroun}\\
{\small surdive@yahoo.fr, fobass1989@gmail.com, ndjeyas@yahoo.fr}
}
\date{}
\maketitle
\begin{abstract} The purpose of this work is to extend the study of the commutative rings whose lattice of ideals can be a structure of BL-algebra as carry out by  Heubo et al in 2018, to non commutative rings appointed in the work as pseudo BL-rings. We study and characterize rings whose ideals form a pseudo BL-algebra, we describe them in terms of their subdirectly irreductible factors. We obtain that these are (up to isomorphism) to a subring of a direct sums of unitary special primary
rings and discrete valuation ring.\\

\par {\bf Keywords:} BL-rings, pseudo-BL algebras, pseudo BL-ring, BL-algebras, multiplication ring. \\

\par {\bf AMS Mathematics Subject Classification(2010):} 03B50, 06D99.
\end{abstract}
\section{Introduction}It is well known that the algebraic study of classical logic proceeds via Boolean algebras \cite{BD09}. For the infinite valued logic of {\L}ukasiewicz, the algebraic analogues of Boolean algebras are the MV-algebras. It is also well known that Boolean algebras can be subsumed within the theory of rings. From their inception there was the question of whether or not the same is true of MV-algebras. In 2009 Belluce and Di Nola wrote an article on commutative rings whose ideals form an MV-algebra \cite{BD09}, in that work they introduced and gave some properties of a class of commutative rings whose lattice of ideals form an MV-algebra, which they called {\L}ukasiewicz rings. An important non-commutative generalization of MV-algebra, known as pseudo MV-algebra was introduced by Georgescu and Iorgulescu \cite{GI99}. The natural question that arises is what happens if one drops the commutativity assumption on {\L}ukasiewicz rings. The answer of that question has been the goal of Kadji, Lele and Nganou in their work {\it A non-commutative generalization of {\L}ukasiewicz rings} \cite{KLN16} where they study and characterize all rings whose ideals form a pseudo MV-algebra, which they called generalized {\L}ukasiewicz rings, this was in 2016. Since the class of BL-algebras contains the MV-algebras, from then in 2018, Heubo et al. \cite{HLNN18} initiated the study of commutative rings whose lattice of ideals can be equipped with a structure of BL-algebra. Recall that BL-algebras where introduced by Peter H\'{a}jek in 1998 as algebraic structures for his Basic Logic to formalize many-valued logics (in short MV logics) induced by continuous $t$-norms on the real unit interval [0,1] and the were generalized to pseudo BL-algebras by Di Nola, Georgescu and Iorgulescu \cite{DGI02}.\\
\indent In the idea to continuous the investigation of classes of rings for which the residuated lattice $A(R)$ of their ideals is an algebra of a well-known subvariety of residuated lattices, in this work we generalize the notion of commutative BL-ring to non commutative case, whom is call in the present pseudo BL-ring.\\
\indent The paper is organized as follows. In section 2, we introduce and recall some definitions and preliminaries results that will be useful in the paper. In section 3, we introduce and study the main properties of pseudo BL-ring, also their connections with some class of non commutative rings. In section 4,  we prove that this class is closed under finite direct products, arbitrary direct sums and homomorphic images. Furthermore, a description of subdirectly irreducible pseudo BL-rings is obtain.

\section{General Preliminaries}
We recall here definitions and properties that are useful throughout the article.
\begin{definition}(\cite{H98})
A BL-algebra is a structure $(A, \vee, \wedge,\odot, 0, 1)$ such that

(1) $(A, \vee, \wedge, 0, 1)$ is a bounded lattice,

(2) $(A, \odot, 1)$ is an abelian monoid, that is, $\odot$ is commutative, associative and $x \odot 1 = 1 \odot x = x$,

(3) the following conditions hold for all $x, y, z \in A$,

\indent (BL-1) $x \odot y \leq z$ iff $x \leq y \rightarrow z$ (residuation),

\indent (BL-2) $x \wedge y = x \odot (x \rightarrow y)$ (divisibility),

\indent (BL-3) $(x \rightarrow y) \vee (y \rightarrow x) = 1$ (prelinearity).
\end{definition}

An MV-algebra is a BL-algebra $A$ satisfying the double negation law: $x^{\ast\ast}=x$, for all $x\in A$, where $x^{\ast}=x\rightarrow0$.

Given any BL-algebra $A$, $MV(A):=\{x^{\ast}~|~x\in A\}$ is an MV-algebra, and is indeed the largest BL-subalgebra of $A$ satisfying the double negation. This MV-algebra is called the MV-center of $A$.\\
We recall that a BL-algebra $A$ is an MV-algebra iff $(x^{\ast})^{\ast}= x$ for any $x \in A$.
\begin{definition}(\cite{GI99})
A pseudo MV-algebra is a structure $(A, \odot, \oplus, (^{\ast}), (^{-}), 0, 1)$ of type (2, 2, 1, 1, 0, 0), such that the following axioms are satisfied for all $x, y, z \in A$,

(1) $x\odot(y\odot z)=(x\odot y)\odot z$,

(2) $x\odot1=1\odot x=x$,

(3) $x\odot0=0\odot x=0$,

(4) $0^{\ast}=1$, $0^{-}=1$,

(5) $(x^-\odot y^-)^{\ast}=(x^{\ast}\odot y^{\ast})^-$,

(6) $x\odot(x^-\oplus y)=y\odot(y^-\oplus x)=(x\oplus y^{\ast})\odot y=(y\oplus x^{\ast})\odot x$,

(7) $x\oplus(x^{\ast}\odot y)=(x\odot y^{\ast})\oplus y$,

(8) $(x^{\ast})^-=x$, where $y\oplus x:=(x^{\ast}\odot y^{\ast})^{-}=(x^{\ast}\odot y^{\ast})^-$.
\end{definition}

\begin{definition}(\cite{DGI02}) A pseudo BL-algebra is a structure $\mathcal{A}=(A, \vee, \wedge, \odot, \rightarrow, \rightsquigarrow, 0, 1)$ of type (2, 2, 2, 2, 2, 0, 0) which satisfies the following axioms, for all $x, y, z \in A$,

(1) $(A, \vee, \wedge, 0, 1)$ is a bounded lattice,

(2) $(A, \odot, 1)$ is a monoid, that is, $\odot$ is associative and $x \odot 1= 1\odot x = x$,

(3) $x \odot y \leq z$ iff $x \leq y \rightarrow z$ iff $y \leq x\rightsquigarrow z$,

(4) $x \wedge y = (x \rightarrow y)\odot x = x \odot (x \rightsquigarrow y)$,

(v) $(x \rightarrow y) \vee (y \rightarrow x) = (x \rightsquigarrow y)\vee (y \rightsquigarrow x) = 1$.
\end{definition}

A ring $R$ is said to be generated by idempotents, if for every $x\in R$, there exists an idempotent element $e\in R$ (that is $e^2=e$) such that $e\cdot x=x\cdot e =x$. There are a residuated lattice form by the two-sided ideals of the ring $R$, this lattice is define by $$A(R):=\langle Id(R),\wedge,\vee,\otimes,\rightarrow,\rightsquigarrow,\{0\},R\rangle,$$ where
$$ I\wedge J=I\cap J,~ I\vee J=I+J,~ I\otimes J=I\cdot J, $$
$$ I\rightarrow J:=\{x\in R|~ xI\subseteq J\},~ I\rightsquigarrow J:=\{x\in R|~Ix\subseteq J\}. $$

Note that $I^{\ast}:=\{x\in R~|~ xI=\{0\}\}$ is simply the right annihilator of $I$ in $R$ and $I^{-}:=\{x\in R~|~ Ix=\{0\}\}$ is simply the left annihilator of $I$ in $R$, where $I$ is an ideal of $R$. Given a ring $R$, recall that an ideal $I$ of $R$ is called an annihilator ideal (resp. a dense ideal) if $I =J^{\ast}$ and $I=K^-$ for some ideals $J$, $K$ of $R$ (resp. $I^{\ast}=I^-=\{0\}$).

$AN^{\ast}(R):=\{I^{\ast}~|~I\in A(R)\}$ and $AN^{-}(R):=\{I^{-}~|~I\in A(R)\}$ denote the sets of left and right annihilator ideals of $R$.

$D^{\ast}(R):=\{I\in A(R)~|~I^{\ast}=\{0\}\}$ and $D^{-}(R):=\{I\in A(R)~|~I^{-}=\{0\}\}$ denote the sets of left and right dense ideals of $R$.

\begin{definition}(\cite{U78})\label{def0.1}
A ring $R$ is called a multiplication ring, if for every ideals $I$, $J$ of $R$ such that $I\subseteq J$, there exist ideals $K$, $K'$ such that $I=J\cdot K=K'\cdot J$.
\end{definition}

\section{Pseudo BL-rings}

We start this section by the introduction of the notion of pseudo BL-rings which shall be non commutative rings for which the lattices of ideals are fit out with pseudo BL-algebra structures. The main properties of these rings are examine and their we settle some connections of these rings to other known classes of rings.

\begin{definition}\label{def1.1}
A ring $R$ generated by idempotents is called a pseudo BL-ring if it satisfies,
\begin{description}
  \item[PBLR-1] $I\cap J=I\cdot(I\rightsquigarrow J)=(I\rightarrow J)\cdot I,$
  \item[PBLR-2] $(I\rightarrow J)+(J\rightarrow I)=(I\rightsquigarrow J)+(J\rightsquigarrow I)=R,$
\end{description}
for all ideals $I$, $J$ of $R$.
\end{definition}

\begin{lemma}\label{lem0}
Let $I$ and $J$ be two ideals of a ring $R$. Then

(1) $I\rightarrow J=I\rightarrow (I\cap J)$ and $I\rightsquigarrow J=I\rightsquigarrow (I\cap J)$.

(2) $(I+J)\rightarrow J=I\rightarrow J$ and $(I+J)\rightarrow I=J\rightarrow I$.
\end{lemma}

\begin{proof}
These are easily derived from the definitions and the operations involved.
\end{proof}

\begin{lemma}\label{lem1}
Let $I$ and $J$ be two ideals of a ring $R$. Then $(I\rightarrow J)\cdot I\subseteq I\cap J$ and $I\cdot(I\rightsquigarrow J)\subseteq I\cap J$.
\end{lemma}

\begin{proof}
 Let $x\in(I\rightarrow J)\cdot I$, then $x=\displaystyle{\sum_{i=1}^na_ib_i}$ with $a_i\in I\rightarrow J$, $b_i\in I$ and $n\in\mathbb{N}$. When using definition of ideal and the definition of $I\rightarrow J$, we  have that $x\in I$ and $x\in J$. Therefore $x\in I\cap J$ and $(I\rightarrow J)\cdot I\subseteq I\cap J$.\\
\indent Let $y\in I\cdot(I\rightsquigarrow J)$, then $y=\displaystyle{\sum_{i=1}^na_ib_i}$ with $a_i\in I$, $b_i\in I\rightsquigarrow J$ and $n\in\mathbb{N}$. When using definition of ideal and the definition of $I\rightsquigarrow J$, we  have that $y\in I$ and $y\in J$. Therefore $y\in I\cap J$ and $I\cdot(I\rightsquigarrow J)\subseteq I\cap J$.
\end{proof}
The following proposition give some axioms equivalent to PBLR-1 and PBLR-2.
\begin{proposition}
Let $R$ be a pseudo BL-ring. For all ideals $I$, $J$, $K$ of $R$, the following hold,

\indent (1) PBLR-1 iff $I\cap J\subseteq (I\rightarrow J)\cdot I$ and $I\cap J\subseteq I\cdot (I\rightsquigarrow J)$.\\
\indent (2) PBLR-2 iff $(I\cap J)\rightarrow K=(I\rightarrow K)+(J\rightarrow K)$ iff $(I\cap J)\rightsquigarrow K=(I\rightsquigarrow K)+(J\rightsquigarrow K)$.\\
\indent (3) PBLR-2 iff $I\rightarrow(J+K)=(I\rightarrow J)+(I\rightarrow K)$ iff $I\rightsquigarrow(J+K)=(I\rightsquigarrow J)+(I\rightsquigarrow K)$.
\end{proposition}

\begin{proof}
(1) By using Lemma \ref{lem1}.

(2) Assume that $(I\cap J)\rightarrow K=(I\rightarrow K)+(J\rightarrow K)$.
Let's take $K=I\cap J$, then $(I\cap J)\rightarrow (I\cap J)=(I\rightarrow (I\cap J))+(J\rightarrow (I\cap J))$. Since $(I\cap J)\rightarrow (I\cap J)=R$, $I\rightarrow (I\cap J)=I\rightarrow J$ and $J\rightarrow (I\cap J)=J\rightarrow I$ by Lemma \ref{lem0}, then $R=(I\rightarrow J)+(J\rightarrow I)$.\\
\indent Assume now that $R=(I\rightarrow J)+(J\rightarrow I)$.

Let $x\in (I\cap J)\rightarrow K$, then $x$ can be write as $x=x\cdot i+x\cdot j$, with $i\in I$, $j\in J$, and $i+j=1$ (1 is an idempotent). Thus, $x\cdot j\in I\rightarrow K$ and $x\cdot i\in J\rightarrow K$. Therefore $(I\cap J)\rightarrow K\subseteq(I\rightarrow K)+(J\rightarrow K)$.

Let $z\in(I\rightarrow K)+(J\rightarrow K)$, then $z=x+y$ with $x\in I\rightarrow K$ and $y\in J\rightarrow K$, that is $xI\subseteq K$ and $yJ\subseteq K$.
Thus $x(I\cap J)\subseteq K$ and $y(I\cap J)\subseteq K$, which allow to obtain $z(I\cap J)=(x+y)(I\cap J)\subseteq K$.

Therefore $(I\rightarrow K)+(J\rightarrow K)\subseteq(I\cap J)\rightarrow K$.

 Using the same previous method with $"\rightsquigarrow"$ it is straightforward that PBLR-2 iff $(I\cap J)\rightsquigarrow K=(I\rightsquigarrow K)+(J\rightsquigarrow K)$.

(3) Assume that PBLR-2 is true, then for ideals $J$ and $K$ of $R$, we have $J\rightarrow K+K\rightarrow J=R$. Let $y\in I\rightarrow (J+K)$, then $y\in R$. This means that $y$ can be write as $y=y\cdot j+y\dot k$, with $j\in J$ and $k\in K$, such that $j+k=1$ (1 is an idempotent). Since $y\cdot j\in I\rightarrow J$ and $y\cdot k\in I\rightarrow K$, then $I\rightarrow(J+K)\subseteq(I\rightarrow J)+(I\rightarrow K)$.

Let $x\in(I\rightarrow J)+(I\rightarrow K)$, then $x=a+b$ with $a\in I\rightarrow J$ and $b\in I\rightarrow K$. Thus $aI\subseteq J$ and $bI\subseteq K$, so $aI+bI\subseteq J+K$. Therefore $xI\subseteq J+K$, and $I\rightarrow(J+K)\supseteq (I\rightarrow J)+(I\rightarrow K)$.

 Conversely, assume now that $I\rightarrow(J+K)=(I\rightarrow J)+(I\rightarrow K)$. Set $I=J+K$, then by Lemma \ref{lem0}, we have $J\rightarrow K+K\rightarrow J=R$ for all ideals $J,K$ of $R$. Thus PBLR-2 holds.

\indent Using the same previous method with $"\rightsquigarrow"$ it is straightforward that PBLR-2 iff $I\rightsquigarrow(J+K)=(I\rightsquigarrow J)+(I\rightsquigarrow K)$.
\end{proof}

\begin{remark}
Let $R$ be a commutative ring. If $I\rightarrow J=I\rightsquigarrow J$, then $R$ is a BL-ring.
\end{remark}

By Definitions \ref{def1.1} and \ref{def0.1}, we have the following results that give a characterization of the pseudo BL-rings.

\begin{proposition}\label{prop1}
Let $R$ be a ring. $R$ satisfies $PBLR-1$ if and only if it is a multiplication ring.
\end{proposition}

\begin{proof}
Let $R$ be a ring.
Assume that $R$ satisfies $PBLR-1$, then $I\cap J=I\cdot(I\rightsquigarrow J)=(I\rightarrow J)\cdot I$. Let $I$ and $J$ be two ideals of $R$ such that $I\subseteq J$. We found two ideals $K$ and $K'$ of $R$, such that $I=J\cdot K=K'\cdot J$.

By $PBLR-1$, $I=I\cap J=J\cdot(J\rightsquigarrow I)=(J\rightarrow I)\cdot J$. Then took $K=J\rightsquigarrow I$ and $K'=J\rightarrow I$.

Conversely, assume that $R$ is a multiplication ring and let $I$, $J$ be ideals of $R$. Since $I\cap J\subseteq J$, there exist two ideals $K$ and $K'$ of $R$ such that $I\cap J=J\cdot K=K'\cdot J$. Hence $J\cdot K\subseteq I$, $K'\cdot J\subseteq I$ and it follows that $K\subseteq J\rightsquigarrow I$ and $K'\subseteq J\rightarrow I$.

Thus $I\cap J\subseteq J\cdot(J\rightsquigarrow I)$ and $I\cap J\subseteq (J\rightarrow I)\cdot I$. Therefore, $PBLR-1$ holds.
\end{proof}

\begin{proposition}\label{prop2}
Let $R$ be a ring. Then the following are equivalent,

\indent (1) $R$ is a pseudo BL-ring.

\indent (2) $A(R)$ is a pseudo BL-algebra.
\end{proposition}

\begin{proof}
 Assume that $R$ is a pseudo BL-ring. Then PBLR-1 and PBLR-2 hold.

By the definition of $A(R)$, $(A(R);\wedge,\vee,\{0\},R)$ is a bounded lattice and $(A(R);\otimes,R)$ is an associative monoid.

Let $I$, $J$ and $K$ be ideals of $R$, then $I\otimes J\subseteq K$ $\Leftrightarrow$ $I\subseteq J\rightarrow K$ $\Leftrightarrow$ $J\subseteq I\rightsquigarrow K$. Since PBLR-1 and PBLR-2 hold, then prelinearity and divisibility hold. Thus $A(R)$ is a pseudo BL-algebra.

 Conversely assume that $A(R)$ is a pseudo BL-algebra.

Since for all $x,y$ in $R$, $x\wedge y=(x\rightarrow y)\cdot x=x\cdot (x\rightsquigarrow y)$ and $(x\rightarrow y)\vee (y\rightarrow x)=(x\rightsquigarrow y)\vee(y\rightsquigarrow x)=1$, then PBLR-1 and PBLR-2 hold. Therefore by Proposition \ref{prop1}, $R$ is a multiplication ring, so generated by idempotents.

Thus $R$ is a pseudo BL-ring.
\end{proof}

We know that if $R$ is a ring and $M_n(R)$ the ring of $n\times n$ matrices over $R$, then any ideal $\mathcal{I}$ of $M_n(R)$ has the form $M_n(I)$ for a uniquely determined ideal $I$ of $R$ \cite{L91}.

\begin{lemma}\label{lem2}
If $R$ is a multiplicative ring, then $M_n(R)$ is a multiplicative ring.
\end{lemma}

\begin{proof}
Let $M_n(I)$ and $M_n(J)$ two ideals of $M_n(R)$ such that $M_n(I)\subseteq M_n(J)$, with $I,J$ two ideals of $R$.

It is easily verify that $M_n(I)\subseteq M_n(J)$ if and only if $I\subseteq J$. Since $R$ is a multiplicative ring, then there exist $K$ and $K'$ two ideals of $R$, such that $I=J\cdot K=K'\cdot J$ which allow to say that $M_n(I)=M_n(J)\cdot M_n(K)=M_n(K')\cdot M_n(J)$. Thus $M_n(R)$ is a multiplicative ring.
\end{proof}

\begin{lemma}\label{lem3}
If $R$ is a multiplicative ring, then $M_n(R)$ satisfies PBLR-1.
\end{lemma}

\begin{proof}
Let $R$ be a multiplicative ring, then by using Lemma \ref{lem2} and Proposition \ref{prop1}, we have that PBLR-1 holds for $M_n(R)$.
\end{proof}

\begin{lemma}\label{lem4}
Let $I$ and $J$ be two ideals of a ring $R$, then

(1) $M_n(I+J)=M_n(I)+M_n(J)$.

(2) $M_n(I\rightarrow J)=M_n(I)\rightarrow M_n(J)$ and $M_n(I\rightsquigarrow J)=M_n(I)\rightsquigarrow M_n(J)$.
\end{lemma}

\begin{proof}
Straightforward, just use the definition of the operations involved.
\end{proof}

\begin{example}
Let $R$ be a Notherian multiplicative ring, then $M_n(R)$ is a pseudo BL-ring.\\
Let $R$ be a discrete valuation ring or a {\L}ukasiewicz ring, then $M_n(R)$ is a pseudo BL-ring.
\end{example}

Now given a ring $R$ and a proper ideal $I$ of R, let's consider the ring $R/I$. An ideal of $R/I$ is a set $J/I=\{x/I~|~x\in J\}$, where $J$ is an ideal of $R$ and $I\subseteq J$.
We also define the left annihilator( respectively the right annihilator) of an ideal $I$ by $I^{\ast}=\{x\in R~|~ xI=\{0\}\}=\{x\rightarrow 0~|~x\in I\}$ (respectively $I^{-}=\{x\in R~|~Ix=\{0\}\}=\{x\rightsquigarrow 0~|~x\in I\}$).

\begin{lemma}\label{lem5}
Let $I$ ,$J$ and $K$ be ideals of a ring $R$, such that $I\subseteq J$ and $I\subseteq K$. Then,

  (1) $(J/I)^{\ast}=(J\rightarrow I)/I$ and $(J/I)^-=(J\rightsquigarrow I)/I$.

  (2) $(J/I)\rightarrow(K/I)=(J\rightarrow K)/I$ and $(J/I)\rightsquigarrow(K/I)=(J\rightsquigarrow K)/I$

\end{lemma}

\begin{proof}
  (1) Let $x/I\in(J/I)^{\ast}$, then for all $a/I$ with $a\in J$, $xa\in I$. This means that $xJ\subseteq I$. Thus $x/I\in (J\rightarrow I)/I$.

  Therefore $(J/I)^{\ast}\subseteq(J\rightarrow I)/I$.

  Conversely let $y/I\in(J\rightarrow I)/I$, then for any $a/I\in J/I$, $ya\in I$ because, $yJ\subseteq I$ by the definition of $y/I\in (J\rightarrow I)/I$. This means that $y/I(J/I)=I$. Therefore $(J/I)^{\ast}\supseteq(J\rightarrow I)/I$.\\
   To show that $(J/I)^-=(J\rightsquigarrow I)/I$ holds, we just have to use same idea as in $(J/I)^{\ast}=(J\rightarrow I)/I$.

  (2) Let $y/I\in (J\rightarrow K)/I$, then $yJ\subseteq K$. This implies that $(xJ)/I\subseteq K/I$ which is $(y/I)\cdot(J/I)\subseteq K/I$. Thus $(J\rightarrow K)/I\subseteq (J/I)\rightarrow(K/I)$.\\
      Conversely let $x/I\in (J/I)\rightarrow(K/I)$, then for all $a\in J$ $xa\in K$, which implies that $x/I\in (J\rightarrow K)/I$. Therefore $(J/I)\rightarrow(K/I)\subseteq (J\rightarrow K)/I$.\\
      With the same idea as for $(J/I)\rightarrow(K/I)=(J\rightarrow K)/I$, we show that $(J/I)\rightsquigarrow(K/I)=(J\rightsquigarrow K)/I$.
\end{proof}

It is possible to observe if the property PBLR-2 in a ring $R$ by observing his quotient by a proper ideal. This is the main purpose of the following propositions.
\begin{proposition}\label{prop2'''}
Let $I$ and $J$ be two ideals of a ring $R$. If the ring $R$ satisfies PBLR-2, then $I\cap J=\{0\}$ implies that $I^{\ast}+J^{\ast}=R$ and $I^-+J^-=R$.
\end{proposition}

\begin{proof}
Assume that $I\cap J=\{0\}$.

PBLR-2 holds means that $I\rightarrow J+ J\rightarrow I=R$ and $I\rightsquigarrow J+ J\rightsquigarrow I=R$. then $I^{\ast}+J^{\ast}=R$ and $I^-+J^-=R$ by using Lemma \ref{lem0}, definitions and operations involved.
\end{proof}
We denote PBLR-3: $I\cap J=\{0\}$ implies that $I^{\ast}+J^{\ast}=R$ and $I^-+J^-=R$.
\begin{proposition}\label{prop2'}
A ring $R$ satisfies PBLR-2 if and only if every quotient of $R$ by an ideal satisfies PBLR-3.
\end{proposition}

\begin{proof}
Let $I$ be an ideal of the ring $R$ which satisfies PBLR-2. Let $J, K$ two ideals of $R$ such that $I\subseteq J$, $J\subseteq K$ and $(J/I)\cap(K/I)=I$. It is clear that $J\cap K=I$. Since $(J/I)^-+(K/I)^-=(J\rightarrow I)/I+(K\rightarrow I)/I=(J\rightarrow(J\cap K))/I+(K\rightarrow(J\cap K))/I=((J\rightarrow K)+(K\rightarrow J))/I=R/I$ (by using Lemmas \ref{lem0}, \ref{lem5}). In the same way, $(J/I)^{\ast}+(K/I)^{\ast}=R/I$. Thus, PBLR-3 holds for $R/I$.\\
\indent Conversely, suppose that in every factor of $R$, PBLR-3 holds. Let $I$, $J$ be ideals of $R$, then $R/(I\cap J)$ satisfies PBLR-3. Clearly  $(I/(I\cap J))\cap(J/(I\cap J))=(I\cap J)$, then $(I/(I\cap J))^-+(J/(I\cap J))^-=R/(I\cap J)$ and $(I/(I\cap J))^{\ast}+(J/(I\cap J))^{\ast}=R/(I\cap J)$ (by using Lemma \ref{lem5}). That is $(I\rightsquigarrow J)/(I\cap J)+(J\rightsquigarrow I)/(I\cap J)=(I\rightarrow J)/(I\cap J)+(J\rightarrow I)/(I\cap J)=R/(I\cap J)$. Thus, $((I\rightarrow J)+(J\rightarrow I))/(I\cap J)=((I\rightsquigarrow J)+(J\rightsquigarrow I))/(I\cap J)=R/(I\cap J)$ that implies $(I\rightarrow J)+(J\rightarrow I)=R$ and $(I\rightsquigarrow J)+(J\rightsquigarrow I)=R$. Finally $R$ satisfies PBLR-2, which conclude the proof.
\end{proof}

\begin{proposition}\label{prop2''}
Every quotient of a multiplication rings is a multiplication ring.
\end{proposition}

\begin{proof}
Let $R$ be a multiplication ring and $K$ an ideal of $R$. Let $I/K$ and $J/K$ be two ideals of $R/K$ such that $I/K\subseteq J/K$, then $I\subseteq J$. Hence, there exist ideals $T$ and $T'$ of $R$ (because $R$ is a multiplication ring) such that $I=J\cdot T=T'\cdot J$. Thus $I/K=(J\cdot T)/K=(T'\cdot J)/K$ and $I/K=J/K\cdot T/K=T'/K\cdot J/K$. Therefore, $R/K$ is a multiplication ring.
\end{proof}

\begin{proposition}\label{prop3}
Let $R$ be a multiplication ring such that the quotient ring satisfies PBLR-3, then $R$ satisfies PBLR-3.
\end{proposition}

\begin{proof}
Let $R$ ba a multiplication ring, then the quotient is a multiplication ring (by Proposition \ref{prop2''}). Since the quotient of $R$ satisfies PBLR-3 then, then the ring $R$ satisfies PBLR-2 (by Proposition \ref{prop2'}). We use the Proposition \ref{prop2'''} to conclude that PBLR-3 hlods for the ring $R$.
\end{proof}

\begin{corollary}\label{cor1}
Every multiplication ring satisfies PBLR-2 if and only if every multiplication ring satisfies PBLR-3.
\end{corollary}

\begin{corollary}\label{cor2}
Any quotient of a pseudo BL-ring by a proper ideal is again a pseudo BL-ring.
\end{corollary}

\begin{definition}(\cite{L91})
A ring $R$ is called a prime ring if the null ideal $(O)$ is a prime ring.
\end{definition}

\begin{proposition}\label{prop4}
Prime ideals of pseudo BL-rings are maximal.
\end{proposition}
\begin{proof}
Let $I$, $J$ be two ideals of a pseudo BL-ring $R$, such that $I$ is a prime ideal and $I\subseteq J\subseteq R$ (with $I\neq J$). Since $I$ is prime, it is known that $R/I$ is a prime ring, so $(J/I)^{\ast}=(J/I)^-=0/I$.

Since $R/I$ is a pseudo BL-ring (by Corollary \ref{cor2}), $(J/I)=((J/I)^{\ast})^{\ast}=((J/I)^{-})^{-}=(0/I)^{\ast}=(0/I)^{-}=R/P$. Thus $J=R$ and $I$ is a maximal ideal of $R$.
\end{proof}

\begin{proposition}\label{prop5}
Let $R$ be a pseudo BL-ring and $I$ an ideal of $R$ such that $I\cap I^{\ast}=\{0\}$ and $I\cap I^-=\{0\}$. Then $I$ is a pseudo BL-ring.
\end{proposition}
\begin{proof}
It is well known that an ideal is a subring, thus $Id(I)\subseteq Id(R)$ because
$$\left\{
                                                                                    \begin{array}{ll}
                                                                                      I\cap I^{\ast}=\{0\} & \hbox{} \\
                                                                                      I\cap I^-=\{0\} & \hbox{}
                                                                                    \end{array}
                                                                                  \right.
$$
 imply that, $$\left\{
                 \begin{array}{ll}
                   I+I^{\ast}=I\vee I^{\ast}=R & \hbox{} \\
                   I+I^-=I\vee I^-=R & \hbox{}
                 \end{array}
               \right.
$$

Since $R$ is generated by idempotents, for any $x\in I$, there are $e,e'\in R$, (right and left idempotents) $m_1, m_3\in I$ and $m_2\in I^{\ast}$, $m_4\in I^-$ such that
$$ex=x, \ e'x=x$$ and $$\left\{
                                 \begin{array}{ll}
                                   m_1+m_2=e & \hbox{} \\
                                   m_3+m_4=e' & \hbox{}
                                 \end{array}
                               \right.
$$
Then $$\left\{
        \begin{array}{ll}
          m_1e+m_2e=e^2=e=m_1+m_2 \Rightarrow m_1-m_1e=m_2e-m_2\in I\cap I^{\ast}=\{0\} &  \\
          e'm_3+e'm_4=(e')^2=e'=m_3+m_4 \Rightarrow m_3-e'm_3=e'm_4-m_4\in I\cap I^{-}=\{0\} &
        \end{array}
      \right.$$

It follow that $$\left\{
                   \begin{array}{ll}
                     m_1=m_1e,~ m_2e=m_2, &  \\
                     e'm_3=m_3,~ e'm_4=m_4, &
                   \end{array}
                 \right.
$$ $x=m_1x$, $x=xm_3$, $(m_1)^2=m_1$, $(m_3)^2=m_3$ since $m_2\in M^{\ast}$, $m_4\in I^-$. Thus $m_1$ and $m_3$ are idempotents in $I$, and $I$ is generated by idempotents.

Since $Id(M)\subseteq Id(R)$, then by Proposition \ref{prop3}, PBLR-2 holds for $I$. Hence $I$ is a pseudo BL-ring.
\end{proof}

\begin{proposition}\label{prop3'}
Let $R$ be a pseudo BL-rings. $R$ are closed under each of the following operations,

(1)finite direct products,

(2) arbitrary direct sums,

(3) homomorphic images.
\end{proposition}
\begin{proof}

$\left(1\right)$ Let $R=\displaystyle{\prod_{k=1}^n} R_{k}$ where each $R_{k}$ is a pseudo BL-ring. Since each $R_{k}$ is generated by idempotents, one gets that any ideal $\mathit{I}$ of $\mathit{R}$ is of the form, $$I=\displaystyle{\prod_{k=1}^n}I_{k},$$ where $I_{k}$ is an ideal of $R_{k}.$

Let $I=\displaystyle{\prod_{k=1}^n}I_{k}$ and $J=\displaystyle{\prod_{k=1}^n}J_{k}$ to be two ideals of $\mathit{R}$, we get,

$$I\cdot J=\displaystyle{\prod_{k=1}^n}I_{k}\cdot\displaystyle{\prod_{k=1}^n}J_{k}
=\displaystyle{\prod_{k=1}^n}I_{k}\cdot J_{k}.$$

Analogously, we also have $I\rightarrow J=\displaystyle{\prod_{k=1}^n}I_{k}\rightarrow J_{k}$
and $I\rightsquigarrow J=\displaystyle{\prod_{k=1}^n}I_{k}\rightsquigarrow J_{k}.$

Then $$I\cap J=\displaystyle{\prod_{k=1}^n}I_{k}\cap\displaystyle{\prod_{k=1}^n} J_{k}=\displaystyle{\prod_{k=1}^n}(I_{k}\cap J_{k})=\displaystyle{\prod_{k=1}^n}(I_{k}\cdot(I_{k}\rightsquigarrow J_{k})=\displaystyle{\prod_{k=1}^n}I_{k}\cdot\displaystyle{\prod_{k=1}^n}(I_{k}\rightsquigarrow J_{k})=I\cdot(I\rightsquigarrow J)$$
and $$\displaystyle{\prod_{k=1}^n}(I_{k}\cap J_{k})=\displaystyle{\prod_{k=1}^n}(I_{k}\rightarrow J_{k})\cdot I_{k}=\displaystyle{\prod_{k=1}^n}(I_{k}\rightarrow J_{k})\cdot \displaystyle{\prod_{k=1}^n}I_{k}=(I\rightarrow J)\cdot I.$$

Thus $R$ satisfies PBLR-1, and because $I+J=\displaystyle{\prod_{k=1}^n}I_{k}+
\displaystyle{\prod_{k=1}^n}J_{k}=\displaystyle{\prod_{k=1}^n}(I_{k}+J_{k}),$
we have, $$(I\rightarrow J)+(J\rightarrow I)=\displaystyle{\prod_{k=1}^n}(I_{k}\rightarrow J_{k})+\displaystyle{\prod_{k=1}^n}(J_{k}\rightarrow I_{k})=\displaystyle{\prod_{k=1}^n}((I_{k}\rightarrow J_{k})+(J_{k}\rightarrow I_{k}))=\displaystyle{\prod_{k=1}^n}R_{k}=R,$$ and $$(I\rightsquigarrow J)+(J\rightsquigarrow I)=\displaystyle{\prod_{k=1}^n}(I_{k}\rightsquigarrow J_{k})+\displaystyle{\prod_{k=1}^n}(J_{k}\rightsquigarrow I_{k})=\displaystyle{\prod_{k=1}^n}((I_{k}\rightsquigarrow J_{k})+(J_{k}\rightsquigarrow I_{k}))=\displaystyle{\prod_{k=1}^n}R_{k}=R.$$ Thus $\mathit{R}$ satisfies PBLR-2. Therefore $\mathit{R}$ is a pseudo BL-ring.

$\left(2\right)$ Let $R=\displaystyle{\bigoplus_{k=1}^n}R_{k},$ where each $R_{k}$ is a pseudo BL-ring. Since each $R_{k}$ is generated by idempotents, one gets that any ideal $I$ of $R$ is of the form, $$I=\displaystyle{\bigoplus_{k=1}^n}I_{k}$$ where $I_{k}$ is an ideal of $R_{k}.$

Let $I=\displaystyle{\bigoplus_{k=1}^n}I_{k}$ and $J=\displaystyle{\bigoplus_{k=1}^n}J_{k}$ to be two ideals of $R,$ we get,
$$I\cdot J=\displaystyle{\bigoplus_{k=1}^n}I_{k}\cdot\displaystyle{\bigoplus_{k=1}^n}J_{k}=\displaystyle{\bigoplus_{k=1}^n}I_{k}\cdot J_{k};$$
analogously, we also have $I\rightarrow J=\displaystyle{\bigoplus_{k=1}^n}I_{k}\rightarrow J_{k}$ and $I\rightsquigarrow J=\displaystyle{\bigoplus_{k=1}^n}I_{k}\rightsquigarrow J_{k}.$ Then $$I\cap J=\displaystyle{\bigoplus_{k=1}^n}I_{k}\cap\displaystyle{\bigoplus_{k=1}^n}J_{k}
=\displaystyle{\bigoplus_{k=1}^n}(I_{k}\cap J_{k})=\displaystyle{\bigoplus_{k=1}^n}(I_{k}\cdot (I_{k}\rightsquigarrow J_{k})=\displaystyle{\bigoplus_{k=1}^n}I_{k}\cdot\displaystyle{\bigoplus_{k=1}^n}(I_{k}\rightsquigarrow J_{k})=I\cdot(I\rightsquigarrow J)$$
and $$\displaystyle{\bigoplus_{k=1}^n}(I_{k}\cap J_{k})=\displaystyle{\bigoplus_{k=1}^n}(I_{k}\rightarrow  J_{k})\cdot I_{k}=\displaystyle{\bigoplus_{k=1}^n}(I_{k}\rightarrow J_{k})\cdot\displaystyle{\bigoplus_{k=1}^n}I_{k}=(I\rightarrow J)\cdot I.$$
Thus $\mathit{R}$ satisfies PBLR-1, and because $I+J=\displaystyle{\bigoplus_{k=1}^n}I_{k}+\displaystyle{\bigoplus_{k=1}^n}J_{k}=\displaystyle{\bigoplus_{k=1}^n}(I_{k}+J_{k}),$
we have, $$(I\rightarrow J)+(J\rightarrow I)=\displaystyle{\bigoplus_{k=1}^n}(I_{k}\rightarrow J_{k})+\displaystyle{\bigoplus_{k=1}^n}(J_{k}\rightarrow I_{k})=\displaystyle{\bigoplus_{k=1}^n}((I_{k}\rightarrow J_{k})+(J_{k}\rightarrow I_{k}))=\displaystyle{\bigoplus_{k=1}^n}R_{k}=R$$ and $$(I\rightsquigarrow J)+(J\rightsquigarrow I)=\displaystyle{\bigoplus_{k=1}^n}(I_{k}\rightsquigarrow J_{k})+\displaystyle{\bigoplus_{k=1}^n}(J_{k}\rightsquigarrow I_{k})=\displaystyle{\bigoplus_{k=1}^n}((I_{k}\rightsquigarrow J_{k})+(J_{k}\rightsquigarrow I_{k}))=\displaystyle{\bigoplus_{k=1}^n}R_{k}=R.$$ Thus $R$ satisfies PBLR-2. Therefore $R$ is a pseudo BL-ring.

$\left(3\right)$ Let $R$ be a PBL-ring and $I$ be an ideal of $R$. We recall that the ideals of $R/I$ are the form $J/I$, where $J$ is an ideal of $R$
with $I\subseteq J.$

Let $J$ and $K$ two ideals of $R$ such that $I\subseteq J,K.$
Then $(J/I)\cap(K/I)=(J\cap K)/I=((J\rightarrow K)\cdot J)/I=(J\rightarrow K)/I\cdot (J/I)=((J/I)\rightarrow (K/I))\cdot (J/I)$; furthermore, $(J\cap K)/I=J\cdot (J\rightsquigarrow K)/I=(J/I)\cdot(J\rightsquigarrow K)/I=(J/I)\cdot(J/I\rightsquigarrow K/I),$ thus $R/I$ satisfies PBLR-1.

We also have $(J/I\rightarrow K/I)+(K/I\rightarrow J/I)=(J+K)/I\rightarrow (K+J)/I=((J\rightarrow K)+(K\rightarrow J))/I=R/I$ and $(J/I\rightsquigarrow K/I)+(K/I\rightsquigarrow J/I)=(J+K)/I\rightsquigarrow (K+J)/I=((J\rightsquigarrow K)+(K\rightsquigarrow J))/I=R/I$  thus $R/I$ satisfies PBLR-2.
Therefore $R/I$ is a pseudo BL-ring.
\end{proof}

We establish the connection between pseudo BL-ring and some well known rings.

\begin{definition}\label{def5}
A unitary ring is called a Von Neumann ring if $R/P$ is a division ring for all prime ideals of $R$.
\end{definition}

\begin{definition}
A non commutative Bear ring is a ring in which for every ideal $I$ of $R$, there exist idempotents $e, e'\in R$ such that $I^{\ast}=eR$ and $I^-=Re'$.
\end{definition}
An ideal in a Bear ring (commutative or not) is called a Bear-ideal if for every $a, b\in R$ such that $a-b\in I$, then $a^{\ast}-b^{\ast}\in I$ and $a^--b^-\in I$.
\begin{definition}(\cite{HLNN18})
A reducing ring is a ring in which 0 is the only nilpotent element, that is, the only element $x\in R$ for which there exists an integer $n\geq1$ such that $x^n=0$.
\end{definition}

\begin{lemma}\label{lem6}
Let $R$ be a ring and $I$, $J$ be ideals of $R$. If $I\cap J=\{0\}$, then $I\subseteq J^{\ast}$, $I\subseteq J^{-}$.
\end{lemma}
\begin{proof}
This is easily derived from the definition of an ideal and the operations $(^{\ast})$ and $(^-)$ involved.
\end{proof}

\begin{proposition}\label{prop6}
In a reduced ring with identities, PBLR-3 holds if and only if this ring is a Bear ring.
\end{proposition}

\begin{proof}
Let $R$ be a reduced ring with identities in which PBLR-3 hlods. Let $K$ be an ideal of $R$, then $K\cap K^{\ast}=\{0\}$ and $K\cap K^-=\{0\}$ implies $K^{\ast}+(K^{\ast})^{\ast}=R$ and $K^{-}+(K^{-})^{-}=R$. Thus, there are $k_1\in K^{\ast}$, $k_1^{'}\in(K^{\ast})^{\ast}$, $k_2\in K^{-}$, $k_2^{'}\in (K^{-})^{-}$ such that $1_L=k_1+k_1^{'}$ and $1_R=k_2+k_2^{'}$.These mean that,
$$\left\{
                    \begin{array}{ll}
                      k_1=k_1\cdot 1_L=k_1(k_1+k_1^{'})=k_1^2 & \hbox{} \\
                      k_2=1\cdot k_2=(k_2+k_2^{'})k_2=k_2^2 & \hbox{}
                    \end{array}
                  \right.
$$ which imply that $k_1$ and $k_2$ are idempotents.

It remains to show that $K^{\ast}=k_1R$ and $K^-=Rk_2$. Let $x\in K^{\ast}$, then $x=1_L\cdot x=(k_1+k_1^{'})x=k_1x\in k_1R$.

Let $y\in k_1R$, then there is $r\in R$ such that $y=k_1r\in K^{\ast}$ because $k_1\in K^{\ast}$. Therefore $K^{\ast}=k_1R$.\\
By the similar way, $K^-=Rk_2$. Then $R$ is a non commutative Bear ring.

Conversely, assume that $I$ and $J$ are two ideals of a Bear ring such that $I\cap J=\{0\}$. We have to show that $I^{\ast}+J^{\ast}=R=I^-+J^-$.

Since $R$ is a Bear ring, there exist $e$ and $e'$ idempotents in $R$ such that $J^{\ast}=eR$ and $J^-=Re'$. Thus $I\subseteq eR$ and $I\subseteq Re'$ (by Lemma \ref{lem6}). Hence for all $i\in I$, $i=er=r'e'$ for some $r,r'\in R$ and $$\left\{
                                                                             \begin{array}{ll}
                                                                               0=er-er=er-e^2r=(1_L-e)er & \hbox {} \\
                                                                               0=r'e'-e'r'=r'e'-(e')^2r'=r'e'(1_R-e') & \hbox{}
                                                                             \end{array}
                                                                           \right.
$$ which imply $1_L-e\in I^{\ast}$ and $1_R-e'\in I^-$.\\
Therefore, $$\left\{
              \begin{array}{ll}
                1_L=(1_L-e)+e\in I^{\ast}+(I^{\ast})^{\ast} & \hbox{} \\
                1_R=(1_R-e')+e'\in I^-+(I^-)^- & \hbox{}
              \end{array}
            \right.
$$

Since $(I^{\ast})^{\ast}$ and $I^-+(I^-)^-$, then $I^{\ast}+J^{\ast}=R$ and $I^-+J^-=R$.
\end{proof}

\begin{proposition}
Every Bear-ideals $I$, $J$ of a Bear ring $R$ satisfy, $(I\rightarrow J)+(J\rightarrow I)=R$ and $(I\rightsquigarrow J)+(J\rightsquigarrow I)=R$.
\end{proposition}

\begin{proof}
Straightforward by using Propositions \ref{prop6} and \ref{prop2'} in this order.
\end{proof}
The same result as in the Proposition \ref{prop2''} holds for Bear rings and Bear-ideals.
\begin{proposition}
Every quotient of a Bear ring (by an Bear ring) is a multiplication Bear ring.
\end{proposition}

\begin{lemma}\label{lem7}
Every Von Neumann is a multiplication ring.
\end{lemma}

\begin{proof}
Let $I$ and $J$ be two ideals of a non commutative Von Neumann ring $R$ and $P$ a prime ideal of $R$. Since all multiplicative ring have the PBLR-1 property and vice-versa (by Proposition \ref{prop1}), then by Definition \ref{def5} of a non commutative Von Neumann ring, we only have to show that $I/P\cap J/P=(I/P)\cdot((I\rightsquigarrow J)/P)$ and $I/P\cap J/P=((I\rightarrow J)/P)\cdot (I/P)$. Since in the division ring $R/P$ the ideals $I/P$ and $J/P$ are trivials, we just have to check four cases.

\indent(Case 1) If $I/P=J/P=\{0/P\}$, then $I/P\cap J/P=\{O/P\}$ and $(I/P)\cdot((I\rightsquigarrow J)/P)=((I\rightarrow J)/P)\cdot (I/P)=\{0/P\}$. Hence the property PBLR-1 holds.

\indent(Case 2) If $I/P=\{0/P\}$ and $J/P=R/P$, then $I=\{0\}$, which imply $I/P\cap J/P=\{0/P\}$ and $(I/P)\cdot((I\rightsquigarrow J)/P)=((I\rightarrow J)/P)\cdot (I/P)=\{0/P\}$. Hence the property PBLR-1 holds.

\indent(Case 3) If $I/P=R/P$ and $J/P=\{0/P\}$, using the definitions of $(I\rightsquigarrow J)/P$ and $(I\rightarrow J)/P$, we easily prove that $(I/P)\cdot((I\rightsquigarrow J)/P)=((I\rightarrow J)/P)\cdot (I/P)=\{0/P\}$. Hence the property PBLR-1 holds.

\indent(Case 4) If $I/P=R/P$ and $J/P=R/P$, using the definitions of $(I\rightsquigarrow J)/P$ and $(I\rightarrow J)/P$, we easily prove that $(I/P)\cdot((I\rightsquigarrow J)/P)=((I\rightarrow J)/P)\cdot (I/P)=R/P$. Hence the property PBLR-1 holds.

we complete the proof.
\end{proof}

\begin{proposition}
Let $R$ be a unitary Von Neumann ring. Then $R$ is a pseudo BL-ring if and only if it satisfies $I/P=J/P=\{0\}$ implies $(I\rightarrow J)/P=R/P$. For all ideals $I$, $J$ of $R$ and all prime ideals $P$ of $R$.
\end{proposition}

\begin{proof}
This result follows directly from the Lemma \ref{lem7}.
\end{proof}

\section{Representation of pseudo-BL algebras}
The purpose of this section is to find a representation of pseudo BL-ring in the sense of subdirectly irreducible algebras. We start this section with the definitions of subdirectly irreducible ring and special primary ring.

\begin{definition}
A ring $R$ is said to be subdirectly irreducible if every subdirect product of $R$ is trivial.
\end{definition}
Equivalently, a ring $R$ is subdirectly irreducible if and only if the intersection of all non-zero ideals of $R$ is non-zero.

\begin{definition}
A ring $R$ is a special primary ring if $R$ has a unique maximal ideal $M$ and if each proper ideal of $R$ is a power of $M$.
\end{definition}

If $P$ is a prime ideal of ring $R$ and if $S = R\setminus P$, we shall denote $0_S^{\ast}$ (respectively $0_S^-$), the $S$-component of the right zero ideal, by $N^{\ast}(P)$ (respectively $N^-(P)$), that is, $N^{\ast}(P) = \{x\in R~|~ xs = 0$  for some  $s\in R\setminus P \}$ (respectively $N^{-}(P) = \{x\in R~|~ sx = 0$ for some $s\in R\setminus P \}$).
\begin{lemma}\label{lem8}
Let $R$ be a ring that is generated by idempotents and $P$ be a prime ideal of $R$. Then, $$\displaystyle{\bigcap_{P}}N^{\ast}(P)=\{0\}~ and~ \displaystyle{\bigcap_{P}}N^{-}(P)=\{0\}.$$
\end{lemma}
\begin{proof}
To show that $\displaystyle{\bigcap_P}N^{\ast}(P)=0$ and $\displaystyle{\bigcap_P}N^{-}(P)=0$, it is sufficient to prove that if $x\neq 0$, then there exists a prime ideal $P$ such that $x\notin N^{\ast}(P)$ and $x\notin N^{-}(P)=0$.

Let $x\neq 0$, then $(xR)^{\ast}\neq R$ and $(Rx)^-\neq R$ (because if we assume that the are all equal, then for all $z\in R$, $z\cdot x=0$ and $x\cdot z=0$ with $x\neq 0$, which mean that for all $z\in R$, $z=0$. So $R=\{0\}$ contradiction). Thus by Proposition [2.10, \cite{HLNN18}] there exists a prime ideal $Q$ of $R$ such that $(xR)^{\ast}\subseteq Q$ and $(Rx)^-\subseteq Q$.

Assume that $x\in N^{\ast}(Q)$ and $x\in N^-(Q)$, then $sx=0$ and $xs=0$ for some $s\notin Q$.

If $s\cdot x=0$, then $(s\cdot x)R=\{0\}$, thus $s\in(xR)^{\ast}\subseteq Q$. Therefore $x\notin N^{\ast}(Q)$.

If $x\cdot s=0$, then $R(x\cdot s)=\{0\}$, thus $s\in(xR)^{-}\subseteq Q$. Therefore $x\notin N^{-}(Q)$.

Hence, $\displaystyle{\bigcap_{P}}N^{\ast}(P)=\displaystyle{\bigcap_{P}}N^{-}(P)=\{0\}$.
\end{proof}

\begin{proposition}\label{prop7}
Every unitary pseudo BL-ring with left and right unit is isomorphic to a subring of a direct product of special primary rings.
\end{proposition}

\begin{proof}
Let $R$ be a pseudo BL-ring with left and right unit, and $P$ a prime ideal of $R$. Then by Proposition \ref{prop1} it is know that $R$ is a multiplicative ring.

Let consider the mappings $f:R\rightarrow \displaystyle{\prod_P}R/P$ such that $f(x)=(\frac{x}{1_R})/P$ a sequence on the prime ideals of $R$ and $g:R\rightarrow \displaystyle{\prod_P}R/P$ such that $g(x)=(\frac{x}{1_L})/P$ a sequence on the prime ideals of $R$. We easily check that $f$ and $g$ are ring homomorphisms. Using the previous Lemma \ref{lem8}, the kernels of $f$ and $g$ are define to be $\ker(f)=\displaystyle{\bigcap_P}N^{\ast}(P)=\{O\}$ and $\ker(g)=\displaystyle{\bigcap_PN^-(P)}=\{0\}$. Hence $R$ is isomorphic to $f(R)$ and $g(R)$ two subring of $\displaystyle{\prod_P}R/P$. Since $R$ is a multiplicative ring, then by Theorem 9.23 and Propositions 9.25 and 9.26 in \cite{LM71}, each $R/P$ is a special primary ring.
\end{proof}

\begin{proposition}\label{prop8}
Every unitary pseudo BL-ring with left and right unit is isomorphic to a subring of a direct product of discrete valuation rings.
\end{proposition}

\begin{proof}
We use the same idea as in Proposition \ref{prop7} and \cite{LM71}.
\end{proof}

\begin{proposition}
Let $R$ be a subdirectly irreducible pseudo BL-ring with minimal ideal $M$. Then,

(1) $M$ is an annihilator ideal.

(2) The annihilator ideals of $R$ are linearly ordered and finite in number.
\end{proposition}

\begin{proof}
(1) We know by definition of $M^{\ast}$ that it is an ideal of $R.$ Since $M$ is the minimal ideal of $R,$ then $M\subseteq M^{\ast},$ and because $M\neq\{0\},$ then $M^{\ast}\neq R.$ Therefore, by the maximality of $M,$ we obtain that $M=M^{\ast};$
similarly, we show that in the case of right ideal $M^{-}$ of $R,$ we also have $M=M^{-};$\newline
therefore $M$ is an annihilator ideal.

(2) Let $R$ be a pseudo BL-ring, then $A(R)$ is a pseudo BL-algebra; are pseudo MV-algebra. Consequently $AN^{\ast}(R)$ and $AN^-(R)$ are respectively the left pseudo $MV-$ center (and right pseudo $MV$-center) of $A(R)$. Moreover, for all ideal $I$ of $R$, we have $(\bigvee I)^{\ast} = \bigcap I^{\ast}$ and $(\bigvee I)^{-} = \bigcap I^{-}$ which imply that every subset of $AN^{\ast}(R)$ and $AN^{-}(R)$ have an infimum. Thus every subset of $AN^{\ast}(R)$ and $AN^{-}(R)$ also have a supremum because if $S$ is a subset of $AN^{\ast}(R)$ (or $AN^{-}(R)$), then $\bigwedge S = (\bigvee S^{\ast})^{\ast}$ (or $\bigwedge S = (\bigvee S^{-})^{-}$). Hence $AN^{\ast}(R)$ and $AN^{-}(R)$ are complete pseudo $MV-$algebras. Also, for every non zero ideal $I$ of $R$, $M\subseteq I$ implies that $I^{\ast}\subseteq M^{\ast}$ and $I^{-}\subseteq M^{-}$.

Therefore, for every proper ideals $J$ in $AN^{\ast}(R)$ and $K$ in $AN^{-}(R)$, we have $J^{\ast}\subseteq M^{\ast}$ and $K^{-}\subseteq M^{-}$ this induce that $AN^{\ast}(R)\setminus\lbrace R\rbrace$ and $AN^{-}(R)\setminus\lbrace R\rbrace$ have both of them a maximum element, namely $M^{\ast}\neq R$ and $M^{-}\neq R$ because $M\neq \{0\}$.

Let $X,Y\in AN^{\ast}(R)$ such that $X$ is not include in $Y$ and $Y$ is not include in $X$, then $X\rightarrow Y, Y\rightarrow X\neq R$ and $X\rightarrow Y, Y\rightarrow X\subseteq M^{\ast};$ hence, by pre-linearity axiom, $R = X\rightarrow Y\vee Y\rightarrow X \subseteq M^{\ast};$ thus $M^{\ast} = R,$ which is a contradiction, so $(AN^{\ast}(R),\subseteq)$ is a pseudo $MV-$chain.\newline
By the same way we prove that $(AN^{-}(R),\subseteq)$ is a pseudo $MV-$chain.
\end{proof}
The next result is the Representation Theorem for pseudo BL-ring.
\begin{theorem}\label{thm1}
Every pseudo BL-ring $R$ is a subdirect product of a family $\{R_x| x\in R\setminus \{0\}\}$ of subdirectly irreductible pseudo BL-rings satisfying,

(1) $A(R_x)$ is isomorphic to $AN^{\ast}(R_x)\oplus D^{\ast}(R_x)$ and to $AN^{-}(R_x)\oplus D^{-}(R_x)$ for all $x\neq0$.

(2) Every ideal of each $R_x$ is either an annihilator ideal or dense.

(3) $A(R)$ is a subdirect product of $\{A(R_x)| x\in R\setminus\{0\}\}$.

(4) $A(R_x)$ is a pseudo BL-algebra with a unique atom.
\end{theorem}
\begin{proof}
Let $R$ be a pseudo BL-ring, it is well known that $R$ is a subdirect product of subdirectly irreductible homomorphic images of $R$. Using the Zorn's lemma, for all $x\in R$ and $x\neq 0$, there is an maximal ideal $K_x$ that do not contain $x$, which implies that $\displaystyle{\bigcap_x}K_x=\{0\}$. Hence each factor $R/K_x$ is subdirectly irreductible and $R$ is the product of the family $\{R/K_x~|~x\in R\setminus \{0\}\}$. Since $R/K_x$ is a pseudo BL-ring by Corollary \ref{cor2} and Proposition \ref{prop3'}. Let pick up the set $\{R_x~|~x\in R\setminus \{0\}\}$ to be $\{R/K_x~|~x\in R\setminus \{0\}\}$. Then the rest of the proof is similar to Theorem 4.2 in \cite{HLNN18}.
\end{proof}
The following corollary is the same as Corollary 4.3 in \cite{HLNN18}.
\begin{corollary}
Let $R$ be a subdirectly irreductible BL-ring. Then

(1) For every annihilator ideal $I\neq R$ and every dense ideal $J$ with $I\subseteq J$, then $J\rightarrow I=I$ and $J\rightsquigarrow I=I$.

(2) For every ideals $I$, $J$ of $R$, either ($I\rightarrow J$ and $I\rightsquigarrow J$) or ($J\rightarrow I$ and $J\rightsquigarrow I$) are denses.
\end{corollary}
\begin{proof}
Let $R$ be a subdirectly irreductible BL-ring. As established in the proof of Theorem \ref{thm1}, $A(R)$ is isomorphic to  $AN^{\ast}(R)\oplus D^{\ast}(R)$ and to $AN^{-}(R)\oplus D^{-}(R)$.

(1) By the definition of the implication in the ordinal sum of hoops, the condition stated hold.

(2) Since $A(R)$ is a pseudo BL-algebra, using the same way as in \cite{BM11} and \cite{DGI02}, we prove that the sets $D^{\ast}(M)$ and $D^-(M)$ of dense elements of any pseudo BL-algebra $M$ is an implicative filter, $M/D^{\ast}(R)$ is isomorphic to $AN^{\ast}(M)$ and $M/D^{-}(R)$ is isomorphic to $AN^{-}(M)$.

Therefore, $A(R)/D^{\ast}(R)$ is isomorphic to $AN^{\ast}(R)$, and $A(R)/D^{-}(R)$ is isomorphic to $AN^{-}(R)$. Futhermore $A(R)/D^{\ast}(R)$ and $A(R)/D^{-}(R)$ are linearly ordered because $AN^{\ast}(R)$ and $AN^{-}(R)$ are linearly ordered. Now using the definitions of order, $\rightsquigarrow$ and  $\rightarrow$ on $A(R)/D^{\ast}(R)$ and $A(R)/D^{-}(R)$ we conclude the proof.
\end{proof}

\section{Conclusion}
It is a well-known result that a $t$-norm has a residuum if and only if the $t$-norm is left-continuous. Therefore, this shows that Basic Logic is not the most general $t$-norm-based logic. In fact, a logic weaker than Basic Logic, called Monoidal $t$-norm-based logic (MTL for short), was defined by Esteva and Godo in \cite{EG01} and  proved to be the logic of left-continuous $t$-norms and their residua. Our research in that way will consist to study commutative rings and non commutative rings $R$ for which $A(R)$ is a MTL-algebra and pseudo MTL-algebra.

\section*{Acknowledgments} This work is supported by the Foundations of Chongqing Municipal Key Laboratory of Institutions of Higher Education ([2017]3), Chongqing Development and Reform Commission (2017[1007]), and Chongqing Three Gorges University.

\end{document}